\documentclass[12pt,a4paper,oneside]{amsart}
\usepackage{amsfonts, amsmath, amssymb, amsthm}
\usepackage[colorlinks=true,citecolor=blue]{hyperref}
\usepackage[margin=1.4in]{geometry}
\usepackage{times}
\usepackage[T1]{fontenc}
\usepackage{bbm}
\usepackage{mathtools}

\usepackage{graphicx}
\usepackage{enumerate}
\usepackage{verbatim}
\usepackage{tikz}

\begingroup
    \makeatletter
    \@for\theoremstyle:=definition,remark,plain\do{%
        \expandafter\g@addto@macro\csname th@\theoremstyle\endcsname{%
            \addtolength\thm@preskip\parskip
            }%
        }
\endgroup

\newtheorem{theorem}{Theorem}[section]
\newtheorem*{theorem*}{Theorem}
\newtheorem{lemma}[theorem]{Lemma}

\theoremstyle{definition}

\newtheorem*{remark*}{Remark}

\newcommand{\conv}{\mbox{conv}}

\begin{document} 

\title{Radon numbers and the fractional Helly theorem}  

\author{Andreas F. Holmsen \and Dong-Gyu Lee}

\date{\today}

\address{\linebreak Andreas F. Holmsen 
\hfill \hfill \linebreak 
Department of Mathematical Sciences  
\hfill \hfill \linebreak
KAIST, 
Daejeon, South Korea.  \hfill \hfill }
\email{andreash@kaist.edu}

\address{\linebreak Dong-Gyu Lee 
\hfill \hfill \linebreak 
Department of Mathematical Sciences  
\hfill \hfill \linebreak
KAIST, 
Daejeon, South Korea.  \hfill \hfill }
\email{ldg2101@kaist.ac.kr}

\begin{abstract} 
A basic measure of the combinatorial complexity of a convexity space is its Radon number. In this paper we show a fractional Helly theorem for convexity spaces with a bounded Radon number, answering a question of Kalai. As a consequence we also get a weak $\varepsilon$-net theorem for convexity spaces with a bounded Radon number. This answers a question of Bukh and extends a recent result of Moran and Yehudayoff.
\end{abstract}

\maketitle

\section{Introduction} 

One of the fundamental statements of combinatorial convexity is Radon's lemma \cite{radon} which says that any set of $d+2$ points in $\mathbb{R}^d$ can be partitioned into two parts whose convex hulls intersect. This property was extended to partitions into $k$ parts, by the celebrated theorem of Tverberg \cite{tverberg}, stating that any set of $(d+1)(k-1)+1$ points in $\mathbb{R}^d$ can be partitioned into $k$ parts whose convex hulls share a common point. There are numerous generalizations, variations, and extensions of these types of results and we refer the reader to the surveys \cite{bara-sob, blagojevic, eckhoff-surv} for more information and further references. 

Radon introduced his lemma in order to prove one of the other fundamental theorems of convexity, namely Helly's theorem \cite{helly}, which states that if the intersection of a finite family of convex sets is empty, then there are  some $d+1$ or fewer sets in the family whose intersection is empty. A far reaching generalization of Helly's theorem is the famous $(p,q)$ theorem due to Alon and Kleitman \cite{pq-alon}, whose proof combined a large number of sophisticated tools and results that had been developed over the years since Helly's original theorem. For more information on the great number of extensions and generalizations of Helly's theorem we refer the reader to \cite{ADLS, eckhoff-surv, eckhoff-pq} and the references therein.

Here we will be concerned with one particular (and important) generalization of Helly's theorem due to Katchalski and Liu \cite{katch-liu} known as the {\em fractional Helly theorem}. It states the following. Let $F$ be a family of $n\geq d+1$ convex sets in $\mathbb{R}^d$, and suppose the number of $(d+1)$-tuples of $F$ with non-empty intersection is at least $\alpha\binom{n}{d+1}$, for some constant $\alpha>0$. Then there are at least $\beta n$ members of $F$ whose intersection is non-empty, where $\beta>0$ is a constant which depends only on $\alpha$ and $d$. 

The fractional Helly theorem plays a crucial role in the proof of the $(p,q)$ theorem (one might even say {\em the} crucial role \cite{akmm}), and various fractional Helly theorems are known \cite{hyperplanes, lattice, eckhoff, kalai-upper}. It is also of considerable interest to understand what conditions can be imposed on a set system which guarantees that it admits the ``fractional Helly property'' (see e.g. \cite{akmm, mato-vc}).

A question in this direction, which we learned from Gil Kalai (personal communication; see also \cite[Problem 18]{kalai-problems}), is whether ``Radon implies fractional Helly''?  Although this may be a (purposefully) vague question, we now describe an axiomatic setting in which it can be made precise. 
\medskip

A {\em convexity space} is a pair $(X,\mathcal{C})$ where $X$ is a (non-empty) set and $\mathcal{C}$ is a family of subsets of $X$ the following properties:
\begin{itemize}
    \item $\emptyset, X\in \mathcal{C}$.
    \item $A,B\in \mathcal{C} \Rightarrow A\cap B\in \mathcal{C}$.
\end{itemize}

For instance, $(\mathbb{R}^d, \mathcal{C}^d)$, where $\mathcal{C}^d$ is the family of all convex sets in $\mathbb{R}^d$, is the standard (Euclidean) convexity. Another typical example is the {\em integer lattice convexity} $(\mathbb{Z}^d, L^d)$ where $L^d = \{\mathbb{Z}^d\cap C  :  C\in \mathcal{C}^d\}$. For an overview of the theory of convexity spaces we refer the reader to the book by van de Vel \cite{vandevel}.

For a general convexity space $(X, \mathcal{C})$ we refer to the members of $\mathcal{C}$ as {\em convex sets}, and in this paper we will make the additional assumption that $\mathcal{C}$ is {\em finite}, that is, we consider only finite convexity spaces. (This does not exclude the standard convexity in $\mathbb{R}^d$ from our results, but simply means that we restrict ourselves to finite families of standard convex sets in $\mathbb{R}^d$. This is not a severe restriction, and the reader should be able replace it by a suitable compactness assumption, but we will keep things finite to emphasize the combinatorial flavor of our results.) 

Given a convexity space $(X, \mathcal{C})$ and a subset $Y\subset X$ we define the {\em convex hull} of $Y$, denoted by $\conv(Y)$, to be the intersection of all the convex sets containing $Y$. This is the minimal convex set containing $Y$.

The main invariant of a convexity space that we will be concerned with is its {\em Radon number}. This is the smallest integer $r_2$ (if it exists) such that every subset $P\subset X$ with $|P|\geq r_2$ can be partitioned into two parts $P_1$ and $P_2$ such that $\conv(P_1)\cap \conv(P_2)\neq \emptyset$. For instance, Radon's lemma states that the Radon number of the standard convexity in $\mathbb{R}^d$ equals $d+2$. (We will only deal with convexity spaces in which $r_2\geq 3$, thereby excluding degenerate/trivial cases.)

\subsection*{Results}
Our main result is a fractional Helly theorem for general convexity spaces with bounded Radon number. This answers Kalai's question. 

\begin{theorem} \label{gen-frac-helly}
For every $r\geq 3$ and $\alpha\in (0,1)$ there exists an $m = m(r)$ and a $\beta=\beta(\alpha, r)\in (0,1)$ with the following property: 
Let $F$ be a family of $n\geq m$ convex sets in a convexity space  with Radon number at most $r$. If at least $\alpha\binom{n}{m}$ of the $m$-tuples of $F$ have non-empty intersection, then there are at least $\beta n$ members of $F$ whose intersection is non-empty.
\end{theorem}

{\em Remark.} Note that  the integer $m$ depends only on $r$ and not on $\alpha$. Our bound on $m$ is quite large in terms of $r$ and is expressed as certain Stirling numbers of the second kind. Here is how we plan to prove Theorem \ref{gen-frac-helly}. First we establish a {\em colorful Helly theorem} for general convexity spaces (Lemma \ref{colorful}), and this is where the integer $m(r)$ appears as the number of colors needed. Next, we consider the ``intersection hypergraph'' carrying the information of which subfamilies of $F$ are intersecting. The colorful Helly theorem may then be interpreted as forbidding certain patterns from the intersection hypergraph (to made precise in section \ref{sec:fractional}), and we can then apply a recent result by the first author \cite{holmsen-arx} concerning the clique number of dense uniform hypergraphs with forbidden substructures. 

\medskip

Now we turn to an application of Theorem \ref{gen-frac-helly}. Recall that 
the {\em transversal number} of a set system $F$ over a set $X$, denoted by $\tau(F)$, is the minimum cardinality of a subset $T\subset X$ such that $T$ intersects every member of $F$. The {\em fractional transversal number} of $F$, denoted by $\tau^*(F)$, is the minimum of $\sum_{x\in X} f(x)$ over all functions $f:X\to [0,1]$ such that $\sum_{x\in S}f(x)\geq 1$ for every $S\in F$. Trivially, we have $\tau^*(F)\leq \tau(F)$, while in general there is no universal bound on $\tau(F)$ in terms of $\tau^*(F)$. Nevertheless, there are non-trivial classes of set systems for which such bounds {\em do} exist, such as hypergraphs with bounded VC-dimension (the $\varepsilon$-net theorem \cite{haussler}), families of convex sets in $\mathbb{R}^d$ (weak $\varepsilon$-nets for convex sets \cite{weak-nets}), and families of convex sets in {\em spearable} convexity spaces with bounded Radon number \cite{moran}.

Our second result shows that $\tau(F)$ can be bounded by a function of $\tau^*(F)$ when $F$ is a family of convex sets in a general convexity space with bounded Radon number. 

\begin{theorem}\label{weak-epsilon}
For every $r \geq 3$ there exists positive constants $c_1=c_1(r)$ and $c_2=c_2(r)$ with the following property:
For any family $F$ of convex sets in a convexity space with Radon number at most $r$, we have $\tau(F)\leq c_1(\tau^*(F))^{c_2}$.
\end{theorem}

{\em Remark.} An equivalent formulation of this result is in terms of weak $\varepsilon$-nets, and it follows that Theorem \ref{weak-epsilon} answers a question of Bukh \cite[Question 3]{bukh}. The weak $\varepsilon$-net theorem for standard convexity in $\mathbb{R}^d$ \cite{weak-nets} is another crucial tool used in Alon and Kleitman's proof of the $(p,q)$ theorem, and it was later shown by Alon, Kalai, Matou{\v s}ek, and Meshulam \cite{akmm} that for abstract set systems, a suitable fractional Helly property will give the type of weak $\varepsilon$-net needed to prove the $(p,q)$ theorem. From this point of view, Theorem \ref{weak-epsilon} is a straight-forward consequence of Theorem \ref{gen-frac-helly} and the work done in \cite{akmm}. The details of this discussion will be given in section \ref{sec:transversals}.

\subsection*{Outline of paper} 

In section \ref{sec:colorful} we establish a colorful Helly theorem for convexity spaces with bounded Radon number (Lemma \ref{colorful}), and use this to prove Theorem \ref{gen-frac-helly} in section \ref{sec:fractional}. In section \ref{sec:transversals} we discuss weak $\varepsilon$-nets and review the main results and concepts from \cite{akmm} needed to prove Theorem \ref{weak-epsilon}.

\subsection*{Notation} We use the following standard notation and terminology. For a natural number $n$, the set $\{1,\dots, n\}$ is denoted by $[n]$, and for a finite set $X$, the set of $k$-tuples ($k$ element subset) of $X$ is denoted by  $\binom{X}{k}$. A {$k$-partition} of  $X$ is a partition of the set $X$ into $k$ non-empty unlabeled parts. The number of $k$-partitions of $[n]$ is denoted by $S(n,k)$. (The numbers $S(n,k)$ are commonly referred to as {\em Stirling numbers of the second kind} \cite[section 1.9]{stanley}.)

\section{A colorful Helly theorem} \label{sec:colorful}

The Radon number of a convexity space can be generalized as follows. For an integer $k\geq 2$, the {\em $k$th partition number} of a convexity space $(X, \mathcal{C})$, denoted by $r_k$, is the smallest integer (if it exists) such that for any multiset $Y\subset X$ with cardinality $r_k$ (counting multiplicities), there exists a $k$-partition of $Y$ into parts $Y_1, \dots, Y_k$ such that $\conv(Y_1)\cap \cdots \cap \conv(Y_k)\neq \emptyset$. Observe that for $k=2$ this indeed coincides with our definition of the Radon number. 

In the case when the ground set $X$ is finite and $k>|X|$ we adopt the convention that $r_k = |X|+1$. In the literature the $k$th partition number is sometimes referred to as the $k$th Radon number or the $k$th Tverberg number, but we will only use the term Radon number when referring to $r_2$. In general, we have the following bound on the $k$th partition number of a convexity space.

\begin{lemma}[Jamison \cite{jamison}] \label{jamsman}
For any integer $k>2$ and convexity space with  bounded Radon number we have, we have $r_k\leq r_2^{\lceil \log_2 k\rceil}$.
\end{lemma}

For certain convexity spaces better bounds are known. For instance, for the standard convexity in $\mathbb{R}^d$, Tverberg's theorem states that the $k$th partition number equals $(d+1)(k-1)+1$. One of the long-standing conjectures concerning the partition numbers of convexity spaces asserted that $r_k\leq (k-1)(r_2-1)+1$, which would imply a purely combinatorial proof of Tverberg's theorem (see e.g. Eckhoff's survey \cite{eckhoff-partition}). However, this conjecture was refuted by Bukh \cite{bukh} who constructed convexity spaces with $r_2=4$ and $r_k\geq 3k-1$, for all $k\geq 3$.

\medskip

The {\em Helly number} of a convexity space $(X,\mathcal{C})$, is the smallest integer $h_{\mathcal{C}}$ (if it exists) such that in any finite family of convex sets whose intersection is empty we can find a subfamily of at most $h_{\mathcal{C}}$ sets whose intersection is empty.  Helly's theorem \cite{helly} states that for the standard convexity in $\mathbb{R}^d$, the Helly number equals $d+1$. In general, we have the following bound on the Helly number of a convexity space. 

\begin{lemma}[Levi \cite{levi}]\label{levi}
For any convexity space, we have $h_{\mathcal{C}}<r_2$.
\end{lemma}

The {\em colorful Helly theorem} discovered by Lov{\'a}sz, and independently by B{\'a}r{\'a}ny \cite{col-hell}, states that if $F_1, \dots, F_{d+1}$ are finite families of convex sets in $\mathbb{R}^d$ such that $\bigcap_{i=1}^{d+1}S_i\neq\emptyset$ for all $S_i\in F_i$ and all $i\in [d+1]$, then for some $i\in[d+1]$ we have $\bigcap_{S\in F_i} S \neq\emptyset$. Note that this implies Helly's theorem by setting $F_1 = \cdots =F_{d+1}$. 

The colorful Helly theorem has many applications in discrete geometry and was originally used by B{\'a}r{\'a}ny (in dual form) to prove the first selection lemma \cite[Theorem 5.1]{col-hell} (see also \cite[chapter 9]{mato}). Later Sarkaria \cite{sarkaria}  showed that it implies Tverberg's theorem (see also \cite{BO} and \cite[chapter 8]{mato}). It should also be noted that the colorful Helly theorem has a topological generalization due to Kalai and Meshulam \cite{top-col-hel}, and an algebraic generalization due to Fl{\o}ystad \cite{floy}.

\medskip

We now establish a colorful Helly theorem for general convexity spaces with bounded Radon number.

\begin{lemma}\label{colorful}
For every integer $r\geq 3$ there exists an integer $m=m(r)$ with the following property:
Let $F_1, \dots, F_m$ be families of convex sets in a convexity space with Radon number at most $r$. If $\bigcap_{i=1}^m S_i \neq \emptyset$ for all $S_i\in F_i$ and all $i\in [m]$, then there exists $1\leq i \leq m$ such that $\bigcap_{S\in F_i}S \neq \emptyset$.
\end{lemma}

\begin{proof}
Let $k=r-1$ and $n=k^{\lceil\log_2r\rceil}$.
We will prove the theorem for $m=S(n,k)$. For contradiction, suppose the families $F_1, \dots, F_m$ satisfy $\bigcap_{S\in F_i}S=\emptyset$ for every $i\in[m]$. From each $F_i$ choose sets $S_i^{(1)}, \dots, S_i^{(k)}$ (with repetitions if necessary) such that $\bigcap_{j=1}^kS_i^{(j)}=\emptyset$, and set 
\[G_i=\{S_i^{(1)}, \dots, S_i^{(k)}\}.\] 
This is possible by definition of the Helly number and Lemma \ref{levi}. 

Let $\mathcal{P}_1, \dots, \mathcal{P}_m$ be the distinct $k$-partitions of $[n]$, which we denote by  
\[\mathcal{P}_i = \{P_i^{(1)}, \dots, P_i^{(k)}\},\] where $P_i^{(1)} \cup \cdots \cup P_i^{(k)} = [n]$. 

For every $t\in [n]$ we define the subfamily $X_t\subset \bigcup_{i=1}^m G_i$ according to the rule
\[S_i^{(j)}\in X_t \iff t\in P_i^{(j)},\]
which implies that $|X_t\cap G_i|=1$ for every $t\in [n]$ and $i\in [m]$. By the hypothesis, we can find a point 
\[x_t\in \textstyle{\bigcap}_{S\in X_t}S\]
for every $t\in [n]$.

By Lemma \ref{jamsman}, we have $n \geq r_k$, and therefore there exists a partition $\mathcal{P}_i$ and a point $x\in X$ such that
\[x\in \conv \{x_t\}_{t\in P_i^{(j)}}\]
for every $j\in [k]$. But this implies that $x\in S_i^{(j)}$ for every $j\in [k]$, which contradicts our initial assumption that $\bigcap_{j=1}^kS_i^{(j)}=\emptyset$.\end{proof}

{\em Remark.} Using elementary bounds on the $S(n,k)$ \cite{rennie} our proof gives a bound on $m(r)$ which is roughly $r^{r^{\lceil \log_2 r\rceil}}$. We have little reason to believe that this bound is optimal, and certainly for specific convexity spaces (such as the standard convexity in $\mathbb{R}^d$) it is very far from the truth.

\section{A fractional Helly theorem} \label{sec:fractional}
Let $H=(V,E)$ be a $k$-uniform hypergraph with vertex set $V$ and edge set $E$. A {\em clique} in $H$ is a subset $S\subset V$ such that $\binom{S}{k}\subset E$, and we let $\omega(H)$ denote the maximum number of vertices of a clique in $H$. For an integer $m \geq k$, let $c_m(H)$ denote the number of cliques in $H$ on $m$ vertices. 

We refer to the set $M=\binom{V}{k}\setminus E$ as the set of {\em missing edges}, and we say that a family $\{\tau_1, \dots, \tau_m\}\subset M$ is a {\em complete $m$-tuple of missing edges} if 
\begin{enumerate}
    \item $\tau_i\cap \tau_j=\emptyset$ for all $i\neq j$, and
    \item $\{t_1,\dots, t_m\}$ is a clique in $H$ for all $t_i\in \tau_i$ and all $i\in [m]$. 
\end{enumerate}

We need the following result \cite[Theorem 1.2]{holmsen-arx} for the proof of Theorem \ref{gen-frac-helly}. It is a generalization of a theorem due to Gy{\'a}rf{\'a}s, Hubenko, and Solymosi \cite{ghs} which deals with the special case $k = m =2$.

\begin{lemma}\label{dk-theorem}
For any $m \geq k >1$ and $\alpha\in (0,1)$, there exists a constant $\beta=\beta(\alpha, k, m) \in (0,1)$ with the following property:
Let $H$ be a $k$-uniform hypergraph on $n$ vertices and $c_m(H) \geq \alpha\binom{n}{m}$. If $H$ does not contain a complete $m$-tuple of missing edges, then $\omega(H)\geq \beta n$.
\end{lemma}

{\em Remark.} For fixed $k$ and $m$ the proof in \cite{holmsen-arx} gives a lower bound on $\beta = \beta(\alpha, k, m)$ which is in $\Omega(\alpha^{k^{(m-1)}})$.

\begin{proof}[Proof of Theorem \ref{gen-frac-helly}]
Let $m = m(r)$ be the function from Lemma \ref{colorful} and set $k=r-1$. For given $\alpha\in (0,1)$ we prove the theorem with $\beta = \beta(\alpha,k,m)>0$, using the function from Lemma \ref{dk-theorem}. 

Define a $k$-uniform hypergraph $H(F,E)$ where $E$ is the set of intersecting $k$-tuples of $F$, that is,
\[E = \left\{ \sigma\in \textstyle{\binom{F}{k}} \; : \; \textstyle{\bigcap}_{S\in \sigma}S\neq \emptyset \right\}.\]
Note that an intersecting $m$-tuple in $F$ corresponds to a clique on $m$ vertices in $H$.
A complete $m$-tuple of missing edges in $H$ corresponds to pairwise disjoint subfamilies $F_1, \dots, F_m$, with $|F_i|= k$, such that 
\[\textstyle{\bigcap}_{S\in F_i}S=\emptyset \; \text{ and } \; \textstyle{\bigcap}_{i=1}^m S_i \neq \emptyset\] for all $S_i\in F_i$ and all $i\in [m]$. By Lemma \ref{colorful} this can not exist, and therefore $H$ does not contain a complete $m$-tuple of missing edges. By the fractional Helly hypothesis we have $c_m(H)\geq \alpha\binom{n}{m}$, and so by Lemma \ref{dk-theorem} we have $\omega(H)\geq \beta n$. This means there exists a subfamily $G\subset F$ with $|G|\geq \beta n$ such that every $k$-tuple of $G$ is intersecting. By Lemma \ref{levi} it follows that $\bigcap_{S\in G}S\neq\emptyset$. 
\end{proof}

\section{Transversal numbers}\label{sec:transversals}

Let $F$ be a finite set system over a set $X$. Given an $\varepsilon\in (0,1)$ and a finite multiset $Y\subset X$, a {\em weak $\varepsilon$-net} for $Y$ (with respect to $F$) is a subset $N\subset X$ such that $N\cap S\neq\emptyset$ for any $S\in F$ with $|S\cap Y|\geq \varepsilon |Y|$ (where elements of $Y$ are counted with multiplicity). 

For the standard convexity in $\mathbb{R}^d$, the {\em weak $\varepsilon$-net theorem} \cite{weak-nets} asserts that any finite multiset $Y\subset \mathbb{R}^d$ admits a weak $\varepsilon$-net (with respect to the standard convex sets) of size at most $f(d,\varepsilon)$. It is a central problem in discrete geometry to understand the correct growth rate of the function $f(d,\varepsilon)$ for fixed $d$ and $\varepsilon \to 0$. It is known that there are sets $Y\in \mathbb{R}^d$ which require weak $\varepsilon$-nets of size $\Omega(\varepsilon^{-1}(\log\varepsilon^{-1})^{d-1})$ \cite{lower-nets}, while the best known upper bound is roughly $\varepsilon^{-d}$ \cite{chazelle}. A recent breakthrough is due to Rubin \cite{natan} who showed $f(d,\varepsilon)\leq \varepsilon^{-(\frac{3}{2}+\delta)}$ for arbitrary small $\delta>0$.

In \cite[Question 3]{bukh}, Bukh asked whether the weak $\varepsilon$-net theorem can be extended to arbitrary convexity spaces. More specifically, does there exist a function $f(r,\varepsilon)$ with the following property: Given a convexity space $(X,\mathcal{C})$ with Radon number at most $r$ and an arbitrary (multi)set $Y\subset X$, does $Y$ admit a weak $\varepsilon$-net (with respect $\mathcal{C}$), where the size of the net is at most $f(r,\varepsilon)$? Bukh himself showed that $f(3,\epsilon)\leq O(\varepsilon^{-2})$ \cite[Proposition 3]{bukh}. 

More recently, Moran and Yehudayoff \cite{moran} considered Bukh's question in the setting of {\em separable} convexity spaces.\footnote{Separable convexity spaces $(X,\mathcal{C})$ are equipped with the additional structure of {\em half-spaces}, i.e. convex sets $H\in \mathcal{C}$ such that $(X\setminus H)\in \mathcal{C}$, and a {\em separation axiom} which requires that for every convex set $S\in\mathcal{C}$ and $x\in (X\setminus S)$ there exists a half-space $H$ such that $S\subset H$ and $x\notin H$.} In this case they showed the existence of weak $\varepsilon$-nets of size at most $(120r^2\varepsilon^{-1})^{{4r^2\ln\varepsilon^{-1}}}$.

\medskip

The relationship between weak $\varepsilon$-nets and transversal numbers is given by the following fact (which follows directly from the definitions). Let $F$ be a finite set system over a set $X$. The following statements are equivalent (with $g(x) = f(\frac{1}{x})$):

\begin{enumerate}
    \item There exists a function $g$ such that for any subsystem $F'\subset F$, we have $\tau(F')\leq g(\tau^*(F'))$.
\item There exists a function $f$ such that for every $\varepsilon\in (0,1)$ and any multiset $Y\subset X$, there is a weak $\varepsilon$-net for $Y$ with respect to $F$ of size at most $f(\varepsilon)$.
\end{enumerate}
By this equivalence, Theorem \ref{weak-epsilon} gives an affirmative answer to Bukh's question.

\medskip

We now review the work of Alon, Kalai, Matou{\v s}ek, and Meshulam \cite{akmm}, in which they investigated the relationship between transversal numbers of set systems and the fractional Helly property. Borrowing their notation, we say that a finite set system $F$ over a set $X$ has property FH$(k,\alpha, \beta)$ if for any subsystem $G \subset F$, with $|G|=n$, in which at least $\alpha\binom{n}{k}$ of the $k$-tuples of $G$ have non-empty intersection, there is an element of $X$ which is contained in at least $\beta n$ members of $G$.

For a finite set system $F$, let $F^\cap$ denote the family of all intersections of the sets in $F$, that is, 
\[F^{\cap} = \{ \textstyle{\bigcap}_{S\in H} S :  H\subseteq F\}.\]
We need the following weak $\varepsilon$-net theorem for abstract set systems due to Alon {\em et al.} 

\begin{theorem*}[\cite{akmm}, Theorem 9] For every $d\geq 1$ there exists an $\alpha>0$ such that the following holds. Let $F$ be a finite family of sets and and suppose $F^{\cap}$ satisfies {\em FH}$(d+1, \alpha, \beta)$ with some $\beta>0$. Then we have \[\tau(F)\leq c_1\cdot(\tau^*(F))^{c_2},\] where $c_1$ and $c_2$ depend only on $d$ and $\beta$.
\end{theorem*}

\begin{proof}[Proof of Theorem \ref{weak-epsilon}] 
Note that for any convexity space $(X,\mathcal{C})$ and any  $F\subset \mathcal{C}$ we have $F^\cap\subset \mathcal{C}$. If the Radon number of $(X, \mathcal{C})$ is at most $r$, then Theorem \ref{gen-frac-helly} implies that there exists an $m=m(r)$ such that for any $\alpha>0$ there is a $\beta>0$ such that property FH$(m,\alpha, \beta)$ holds for any subfamily of $\mathcal{C}$, in particular for any $F^\cap$. Our theorem therefore follows from the weak $\varepsilon$-net theorem for abstract set systems \cite[Theorem 9]{akmm}. \end{proof}

\medskip

{\em Remark.} It would be interesting to find further properties of set systems which guarantee a weak $\varepsilon$-net theorem, and some directions are suggested by Moran and Yehudayoff \cite[section 6]{moran}. Finally, let us point out that our results also imply a $(p,q)$ theorem in convexity spaces with bounded Radon number. This follows immediately from the results in \cite{akmm}. (We leave the details to the reader.)

\begin{theorem}
Let $r\geq 3$ and let $m=m(r)$ be the value from Theorem \ref{gen-frac-helly}. For any $p\geq q \geq m$ there exists a constant $c= c(p,q)$ with the following property:
Let $F$ be a family of convex sets in a convexity space with Radon number at most $r$, and suppose among any $p$ members of $F$ there are some $q$ of them with non-empty intersection. Then $\tau(F)\leq c$.
\end{theorem}

\end{document}